\documentclass[journal,twoside]{IEEEtran}
\usepackage[utf8]{inputenc}
\usepackage{cite}
\usepackage{graphicx}
\usepackage{url}
\usepackage{lipsum}
\usepackage{color}
\usepackage{xcolor} 
\usepackage{amsmath}
\usepackage[english]{babel}
\usepackage{amsthm}
\newtheorem{theorem}{Theorem}[section]

\newtheorem{proposition}[theorem]{Proposition}
\theoremstyle{definition}

\usepackage{mathtools}
\usepackage{array}
\usepackage{multirow}
\usepackage{makecell}
\usepackage{float}
\usepackage{setspace}
\usepackage{supertabular,booktabs}
\usepackage{amssymb}
\usepackage{algorithm}
\usepackage{physics}
\usepackage[shortlabels]{enumitem}
\usepackage{algpseudocode}
\usepackage{comment}
\newboolean{showcomments}
\setboolean{showcomments}{true}

\usepackage{amsfonts}
\usepackage{multicol}

\allowdisplaybreaks[4]
\usepackage{authblk}
\usepackage{nomencl}
\usepackage[symbol]{footmisc}

\usepackage{subcaption}
\newcommand{\changed}[1]{\textcolor{black}{#1}}
\newcommand{\julio}[1]{\ifthenelse{\boolean{showcomments}}
        { \textcolor{red}{(JB:  #1)}}{}}
\newcommand{\brian}[1]{\ifthenelse{\boolean{showcomments}}
        { \textcolor{blue}{(BL:  #1)}}{}}
\makenomenclature
\usepackage[colorlinks=true,linkcolor=black,anchorcolor=black,citecolor=black,urlcolor=black]{hyperref}

\def\BibTeX{{\rm B\kern-.05em{\sc i\kern-.025em b}\kern-.08em
    T\kern-.1667em\lower.7ex\hbox{E}\kern-.125emX}}
\begin{document}

\title{\LARGE Linear OPF-based Robust Dynamic Operating Envelopes with Uncertainties in Unbalanced Distribution Networks}
\author{Bin Liu,~\IEEEmembership{Member,~IEEE},~Julio H. Braslavsky,~\IEEEmembership{Senior Member,~IEEE}~and~Nariman Mahdavi
\thanks{Manuscript received: May 10, xxxx; revised: May 20, xxxx; accepted: May 30, xxxx. Date of CrossCheck: May 30, xxxx.}
\thanks{This work is supported by the CSIRO Strategic Project on Network Optimisation \& Decarbonisation under Grant Number: OD-107890.}
\thanks{Bin Liu (\textit{corresponding author}) was with Energy Centre, Commonwealth Scientific and Industrial Research Organisation (CSIRO), Newcastle 2300, Australia. He is now with the Network Planning Division, Transgrid, Sydney 2000, Australia. (eeliubin@ieee.org)}
\thanks{Julio, H. Braslavsky and Nariman Mahdavi are with the Energy Systems Program, Energy Centre, Commonwealth Scientific and Industrial Research Organisation (CSIRO), Newcastle 2300, Australia. }
\thanks{DOI: 10.35833/MPCE.2023.000653}
}

\maketitle
\begin{abstract}
	Dynamic operating envelopes (DOEs), as a key enabler to facilitate DER integration, have attracted increasing attention in the past years. However, uncertainties, which may come from load forecast errors or inaccurate network parameters, have been rarely discussed in DOE calculation, leading to compromised quality of the hosting capacity allocation strategy. This letter studies how to calculate DOEs that are immune to such uncertainties based on a linearised unbalanced three-phase optimal power flow (UTOPF) model. With uncertain parameters constrained by norm balls, formulations for calculating Robust DOEs (RDOEs) are presented along with discussions on their tractability. Two cases, including a 2-bus illustrative network and a representative Australian network, are tested to demonstrate the effectiveness and efficiency of the proposed approach.
\end{abstract}

\begin{IEEEkeywords}
	Distributed energy resources (DERs), dynamic operational envelopes (DOEs), feasible region, robust optimisation, uncertainty modelling, unbalanced optimal power flow
\end{IEEEkeywords}

\section{Introduction}
\IEEEPARstart{T}{he} penetration of distributed energy resources (DERs) has been rapidly increasing worldwide in the past years, leading to a series of issues that require close coordination among transmission system operators (TSOs), distribution system operators (DSOs) and emerging DER aggregators via virtual power plants (VPPs) \cite{DEIP2022,EnergyNetworksAustralia2020}.
Dynamic operating envelope (DOE), which specifies the operational range for customers with DERs at the connection point that is permissible within the network operational limits, is identified as a key enabler in  future power system architectures and has gained increasing interest from both industry and academia to manage DER export/import limits and to facilitate DER participation in electricity markets \cite{EnergyNetworksAustralia2020}. \changed{Differing from static operating envelopes (SOEs) that are calculated based on worst operational scenarios that occur rarely in a distribution network, DOEs can be updated more frequently (day ahead, every several hours, hourly or every 15 minutes) to avoid unnecessary limitations on DER integration and free up the latent network hosting capacity. Moreover, compared with the co-optimisation of scheduled generators, DERs and network operations, DOEs can be calculated and published by an individual distribution system operator (DSO) to avoid traceability issues that may arise when both transmission and medium/low-voltage distribution networks are modelled and optimised together by a single central system operator.}

Although substantial advances have been made in developing approaches to calculating DOEs in recent years \cite{DEIP2022,Petrou2021,BL_ieee_access}, uncertainties, which may arise from load forecasting and inaccurate network parameters, are typically ignored in the calculations, which may lead to unreliable DOE allocations.
To address this issue, this letter proposes a method to calculate robust DOEs (RDOEs) that are immune to such uncertainties.
It is noteworthy that DOEs are inherently linked to the concept of feasible region (FR), which has been discussed for transmission networks in \cite{Wei2015} and for distribution networks in \cite{Riaz2022}. Geometrically, each DOE allocation strategy is linked to a feasible point on the boundary of the FR when it is calculated by a deterministic approach, such as the ones proposed in \cite{BL_ieee_access,Petrou2021}. 
\color{black} Contributions of the paper are summarised as follows.
\begin{enumerate}
    \item The formulation of the FR for DERs, along with its appropriate reformulation, is presented based on a linearised unbalanced three-phase optimal power flow (UTOPF) model. Formulating the FR first is for the convenience of considering uncertainties from network impedances and/or forecast errors. 
    \item The robust feasible region (RFR), which is a variation of the FR, while considering the studied uncertainties modelled as norm inequalities, is presented based on static robust optimisation theory, leading to deterministic convex formulations for calculating RDOEs.
\end{enumerate}

The proposed approach is tested and demonstrated efficiently on an illustrative network and a representative Australian network.

\color{black}
\section{Calculating DOEs via Deterministic UTOPF}
Based on UTOPF, a deterministic approach to calculating DOEs can be formulated as
\begin{subequations}\normalsize
	\label{oe_utopf}
	\begin{eqnarray}
		\label{oe_utopf-obj}
		\max_{P_m,Q_m}{r(P)}\\
		\label{oe_utopf-01}
		V^{\phi}_{i_\text{ref}}=V^{\phi}_{0}~~\forall \phi\\
		\label{oe_utopf-02}
		V^{\phi}_{i}-V^{\phi}_{j}=\sum\nolimits_\psi{z_{ij}^{\phi\psi}I^{\phi}_{ij}}~~\forall \phi,\forall ij\\
		\label{oe_utopf-03}
		\sum_{n:n\rightarrow i}{I^{\phi}_{ni}}-\sum_{m:i\rightarrow m}{I^{\phi}_{im}}=\sum_m{\frac{\mu_{\phi,i,m}(P^\text{}_{m}-\text{j}Q^\text{}_m)}{(V^\phi_{i})^*}}\nonumber\\~~\forall \phi,\forall i\neq i_\text{ref}\\
		\label{oe_utopf-05}
		V^\text{min}_{i}\le |V^{\phi}_{i}|\le V^\text{max}_{i}~ \forall \phi,\forall i
	\end{eqnarray}
\end{subequations}
where
$r(P)$ is the objective function, reflecting the efficiency and fairness in calculating DOEs, and can be in linear or convex quadratic forms.
$i_\text{ref}$ is the index of the reference bus and $V^{\phi}_{0}$ is its fixed voltage at phase $\phi$ (known parameter).
$V^{\phi}_{i}~(i\neq 0)$ is the voltage of phase $\phi$ at node $i$.
$I^{\phi}_{ni}$ is the current in phase $\phi$ of line $ni$: flowing from bus $n$ to bus $i$.
$P_m$ is the active power demand of customer $m$ while $Q_m$ is the reactive power demand. For simplicity, all $P_m$ and $Q_m$ are treated as variables in the formulation. However, they will be fixed to their forecasted values if they are uncontrollable.
$\mu_{\phi,i,m}\in\{0,1\}$ is a parameter indicating the phase connection of customer $m$ with its value being $1$ if it is connected to phase $\phi$ of bus $i$ and being $0$ otherwise.
$V^\text{min}_{i}$ and $V^\text{max}_{i}$ are the lower and upper limits of $|V_{i}|$, respectively.

In the formulation, the objective function aims at maximising $r(P)$ to obtain the desired DOEs, subject to \eqref{oe_utopf-01} specifying the voltage at the reference bus, \eqref{oe_utopf-02} formulating voltage drop in each line, \eqref{oe_utopf-03} assuring Kirchhoff's current law is satisfied, and \eqref{oe_utopf-05} representing voltage magnitudes constraints. It is noteworthy that only voltage magnitude constraints are considered in this letter and, however, other constraints can be conveniently incorporated.

Noting that for most distribution networks, differences of voltage angles in each phase are sufficiently small \cite{RN38} and nodal voltages throughout the network are around $1.0~p.u.$, \eqref{oe_utopf-03} can be linearised by fixing $V^\phi_i$ in the denominator on the right-hand side of \eqref{oe_utopf-05}, leading to a compact formulation of \eqref{oe_utopf} with linear constraints as
\begin{subequations}\normalsize\label{oe_utopf_com}
	\begin{eqnarray}
		\label{oe_utopf_com-obj}
		\max_{p_1,q_1}{r(p_1)}\\
		\label{oe_utopf_com-01}
		[A_1,A_2][p_1^T,p_2^T]^T+[B_1,B_2][q_1^T,q_2^T]^T+Cl=b\\
		\label{oe_utopf_com-02}
		Dv+El=d\\
		\label{oe_utopf_com-03}
		Fv\le f
	\end{eqnarray}
\end{subequations}
where
$p_1$ and $p_2$ are vectors related to active powers from active customers (VPP participants) and passive customers (the customers for which active powers need to be forecasted or estimated), respectively;
$q_1$ and $q_2$ are vectors consisting of reactive powers that are controllable, and that need to be forecasted or estimated, respectively;
$l$ and $v$ represent the vectors consisting of state variables related to line currents and nodal voltages, respectively;
$A=[A_1,A_2],B=[B_1,B_2],C,b,D,E,d,F$ and $f$ are constant parameters with appropriate dimensions.

\changed{It is noteworthy that the fixed value of $V^\phi_i$, say $\bar V^\phi_i$ can be estimated, for example, as $1.0\angle 0^\circ ~p.u.$, $1.0\angle 120^\circ ~p.u.$ and $1.0\angle -120^\circ ~p.u.$ for phase $a,b$ and $c$, or acquired from measurements from the network, to improve the accuracy of the linearised formulation further. More discussions on the linearisation accuracy will be presented and discussed in Section \ref{case_study}.}

In the formulation,
\eqref{oe_utopf_com-01} links back to \eqref{oe_utopf-03} after linearisation and represents the relations between line currents ($l$) and residential demands ($p_1$, $p_2$, $q_1$ and $q_2$);
\eqref{oe_utopf_com-02} represents the linearised power flow equations that link the bus voltages ($v$) and currents running in all lines ($l$), i.e. \eqref{oe_utopf-01}-\eqref{oe_utopf-02},
and \eqref{oe_utopf_com-03} represents all the operational constraints after the linearisation, i.e \eqref{oe_utopf-05}.

Noting that only $p_1$ and $q_1$ are independent variables, \eqref{oe_utopf_com} defines the FR, as a function of $q_1$, for $p_1$. Therefore, if all realised values of $p_1$ fall within the FR, the integrity of the network can be guaranteed. After removing state variables $v$ and $l$, the FR for $p_1$ can be expressed as the following polyhedron.
\begin{eqnarray}\label{fr}
	\mathcal{F}(q_1)=\left\{p_1\left\vert\begin{matrix}
		FD^{-1}EC^{-1}(A_1p_1+A_2p_2+B[q_1^T,q_2^T]^T\\-b) \le f-FD^{-1}d \\
	\end{matrix}\right.\right\}
\end{eqnarray}

\changed{It is noteworthy that both $C$ and $D$ can be proved to be invertible since both of them can be constructed from the connectivity matrix of all buses (excluding the reference bus) and all lines in a distribution network with radial topology.} Further, for the convenience of later discussions, we have the following proposition and its proof. 

\begin{proposition}\label{MTTC}
	The FR expressed as \eqref{fr} is equivalent to 
	\begin{eqnarray}\label{fr_new}
		\mathcal{F}(q_1)=\left\{p_1\left\vert\begin{array}{l}
			vec(E)^TH_i(A_1p_1+A_2p_2+B[q_1^T,q_2^T]^T\\-b)\le t_i~\forall i \\
		\end{array}\right.\right\}
	\end{eqnarray}
    where $vec(\cdot)$ is the vectorising operator for a matrix. \changed{For example, for $H=\begin{bmatrix}h_{11} & h_{12} \\ h_{21} & h_{22}\end{bmatrix}$, we have $vec(H)=[h_{11}, h_{21}, h_{12}, h_{22}]^T$.}
	$[\cdot]_i$ indicates the $i^{th}$ row of a matrix or the $i^{th}$ element of a vector.
	$H_i=C^{-1}\otimes ([F]_iD^{-1})^T$ with $\otimes$ being the Kronecker product, and $t_i=[f-FD^{-1}d]_i$.
\end{proposition}

\begin{proof}
	From \eqref{fr}, it is obvious that the $i^{th}$ inequality expression is
	\begin{eqnarray}
		[F]_iD^{-1}EC^{-1}(A_1p_1+A_2p_2+B[q_1^T,q_2^T]^T-b)
			\nonumber\\\le [f-FD^{-1}d]_i=t_i
	\end{eqnarray}
	For the term $[F]_iD^{-1}EC^{-1}$, we have \cite{matrixB_book}
	\begin{eqnarray}
		\label{proof_01}
		vec([F]_iD^{-1}EC^{-1})=(C^{-T}\otimes [F]_iD^{-1})vec(E)
	\end{eqnarray}
	which leads to 
	\begin{eqnarray}
		\label{proof_03}
		vec(E)^T(C^{-1}\otimes ([F]_iD^{-1})^T)w=[F]_iD^{-1}EC^{-1}w
	\end{eqnarray}
	with $w=(A_1p_1+A_2p_2+B[q_1^T,q_2^T]^T-b)$, which proves the proposition. 
\end{proof}


Therefore, seeking DOEs through the deterministic approach with controllable $q_1$ is equivalent to solving
\begin{eqnarray}\label{max_oe}
    \max\nolimits_{(p_1,q_1)}\{r(p_1)|s.t.~p_1\in\mathcal{F}(q_1)\}
\end{eqnarray}
and one typical formulation of the objective function, which will be used in this paper later, is $r(p_1)=1^Tp_1$.

\section{Robust Operating Envelopes}
\subsection{Uncertainty Modelling}
Comparing \eqref{oe_utopf} and \eqref{oe_utopf_com}, errors in forecasting $P_m$ for passive customers and uncontrollable $Q_m$, and inaccuracies in $z^{\phi\psi}_{ij}$ will lead to uncertainties in $p_2,q_2$ and $E$, respectively. In this letter, such uncertainties are formulated as
\begin{subequations}\normalsize\label{uset}
	\begin{eqnarray}
		\label{uset_E}
		\mathcal{E}=\mathcal{E}_1\cap\mathcal{E}_2=\{E|vec(E)=e_1+J_1x, \norm{x}_\infty\le \gamma_1\}\nonumber\\
		\cap\{E|vec(E)=e_2+ J_2x, \norm{x}\le \gamma_2\}\\
		\label{uset_p}
		\mathcal{P}=\mathcal{P}_1\cap\mathcal{P}_2=\{p_2|p_2= u_1+ U_1y, \norm{y}_\infty\le \rho_1 \}\nonumber\\
		\cap\{p_2|p_2=u_2+ U_2y, \norm{y}\le \rho_2 \}\\
		\label{uset_q}
		\mathcal{Q}=\mathcal{Q}_1\cap\mathcal{Q}_2=\{q_2|q_2=w_1+ W_1z, \norm{z}_\infty\le \theta_1 \}\nonumber\\
		\cap\{q_2|q_2=w_2+ W_2z, \norm{z}\le \theta_2 \}
	\end{eqnarray}
\end{subequations}
where $e_i,u_i,w_i,J_i,U_i,W_i,\gamma_i,\rho_i$ and $\theta_i$ are constant parameters describing the uncertainty sets, and $x,y,z$ are random variables;
The $\infty$-norm constraint in $\mathcal{E}_1,\mathcal{P}_1$ and $\mathcal{Q}_1$ provides a general lower/upper bound for the random variable while the norm constraint in $\mathcal{E}_2,\mathcal{P}_2$ and $\mathcal{Q}_2$, which can take 1/2/$\infty$-norm or other types of norms, is to further reduce the conservativeness of the uncertainty set.

Several remarks on uncertainty modelling are given below.
\begin{enumerate}
	\item Constant parameters can be chosen depending on the physical truth or historical error distributions. For example, if $E$ is usually within 10\% error of $\bar E$, where $\bar E$ is the nominal value of $E$, we can set $\mathcal{E}=\mathcal{E}_1$, $e_1=vec(\bar E)$, $J_1=diag(e_1)$, and $\gamma_1=10\%$. As another example, if $p_2$ falls in $[0,p^\text{max}_2]$ and its forecast error follows a multivariate normal distribution with expectation and covariance being $0$ and $\Sigma$ respectively and 2-norm is used in $\mathcal{P}_2$, $u_1$ can be set as a vector with all its elements being $p^\text{max}_2/2$, $U_1=diag(u_1)$ and $\rho_1=1$ for $\mathcal{P}_1$. Noting that $y^T\Sigma y$ follows a \emph{Chi-square} distribution with freedom degrees of $n$, i.e. $y\sim\chi^2_n$, we can set $u_2=\bar p_2$ in $\mathcal{P}_2$ with $\bar p_2$ being the forecasted value of $p_2$, $U_2=diag(u_2)$ and $\rho_2=(\chi^2_{n,1-\epsilon})^{1/2}$ to guarantee that $p_2$ now falls within $\mathcal{P}_2$ with a confidence level of $1-\epsilon$.
	\item Constant parameters can also be chosen depending on the confidence level of satisfying \eqref{fr_new}, leading to equivalent chance-constrained optimisation problems. This, however, is beyond the scope of this letter, and interested readers are referred to \cite{Bertsimas2022} for more discussions.
	\item $\mathcal{E}$, $\mathcal{P}$ and $\mathcal{Q}$ can be formulated as other types of convex sets, which, however, may affect the tractability of the formulated problem if two or more uncertainties co-exist. More discussions will be provided in the next section.
\end{enumerate}

\subsection{Robust DOEs}
For the convenience of discussion, we here assume both $q_1$ and $q_2$ are controllable, thus removing uncertainties in $q_2$. However, similar to dealing with uncertainty in $p_2$, the proposed approach can be easily extended to the case when uncertainty in $q_2$ exists.

Since the optimisation problem \eqref{max_oe} only contains linear inequality constraints \eqref{fr_new}, the essential idea in seeking RDOEs is to make sure \eqref{fr_new} is always satisfied for any realisation of uncertain parameters. To get the robust counterpart (RC) of \eqref{max_oe}, the equivalent reformulation of \eqref{fr_new}, considering the uncertainties that are bounded by \eqref{uset}, should be derived. Taking a generic formulation $f(\varepsilon,x)\le 0$, where $x$ is a variable and $\varepsilon$ is an uncertain parameter belonging to $\mathcal{E}=\{g(\varepsilon)\le 0\}$, as an example, its RC formulation is 
\begin{eqnarray}\label{dual_01}
\max_{\varepsilon\in\mathcal{E}}{f(\varepsilon,x)}=\min_{\alpha\ge 0}\max_{\varepsilon}{f(\varepsilon,x)-\alpha g(\varepsilon)}\le 0
\end{eqnarray}

With the following fact or assumption, \eqref{dual_01} can then be reformulated as deterministic linear or other convex constraints. 
\begin{enumerate}
	\item If the \emph{min} operator is on the left-hand side of a \emph{less-than-or-equal-to} constraint, it can be safely removed. For example $h(x,\beta)\le 0$ can always guarantee that $\min_\beta{h(x,\beta)}\le 0$.
	\item Under certain circumstances (for example, $f(\varepsilon,x)$ being linear and $\mathcal{E}$ being norm constraints, i.e. $\mathcal{E}=\{\varepsilon|\norm{\varepsilon}\le \bar\varepsilon\}$), $\max_{\varepsilon}{f(\varepsilon,x)-\alpha g(\varepsilon)}$ can be expressed in an equivalent deterministic form without the \emph{max} operator.
\end{enumerate}

Next, we will discuss how such reformulation techniques can be applied in deriving the RDOE formulation under various uncertainty models. 

\subsubsection{With uncertainty only in E}\label{sec_uset_E}
\color{black}
Fixing $p_2$ at $\bar p_2$ and denoting $h_i=H_i(A_1p_1+A_2\bar p_2+Bq- b)$, the inequality expression in \eqref{fr_new} with any realisation of uncertain $E$ is equivalent to
\begin{eqnarray}\label{case_1_1}
    \max_{E\in\mathcal{E}}{vec(E)^T H_i(A_1p_1+A_2\bar p_2+Bq- b)}\le t_i
\end{eqnarray}

For the left-hand side of \eqref{case_1_1}, we further have
\begin{subequations}\normalsize
    \begin{eqnarray}
		\max_{E\in\mathcal{E}}{vec(E)^T H_i(A_1p_1+A_2\bar p_2+Bq- b)}=\delta^*(h_i|\mathcal{E})\nonumber\\
		=\min_{\tau_{i1},\tau_{i2}}\{\delta^*(\tau_{i1}|\mathcal{E}_1)+\delta^*(\tau_{i2}|\mathcal{E}_2))\big|\tau_{i1}+\tau_{i2}=h_i\}\nonumber\\
		\label{rfr_E_delta}
		=\min_{\tau_{i1},\tau_{i2}}\{\sum\nolimits_j(e_j^T\tau_{ij}+\delta^*(J_j^T\tau_{ij}|\mathcal{X}_j))\big|\sum\nolimits_j\tau_{ij}=h_i\}\\
		\label{rfr_E}
		=\min_{\tau_{i1},\tau_{i2}}\{\sum\nolimits_j e_j^T\tau_{ij}+\gamma_1||J_1^T\tau_{i1}||_1+\gamma_2||J_2^T\tau_{i2}||_*)\nonumber\\\big|\sum\nolimits_j\tau_{ij}=h_i\}
	\end{eqnarray}
\end{subequations}
where $\delta^*(y|\mathcal{X})=\sup_{x\in\mathcal{X}}{y^Tx}$ is the conjugate function of the support function $\delta(x|\mathcal{X})$, $\mathcal{X}_1=\{x|\norm{x}_\infty\le\gamma_1\}$, $\mathcal{X}_2=\{x|\norm{x}\le\gamma_2\}$ and $||\cdot||_*$ represents the dual norm operator. Moreover, $\delta^*(y|\mathcal{X})$ is always a convex function \cite{cvxbook}.

After safely removing the \emph{min} operator in \eqref{rfr_E}, \eqref{case_1_1} can be reformulated as \cite {Bertsimas2022}
\begin{subequations}\small\label{case_1_2}
	\begin{eqnarray}
		\sum\nolimits_j e_j^T\tau_{ij}+\gamma_1||J_1^T\tau_{i1}||_1+\gamma_2||J_2^T\tau_{i2}||_*)\le t_i\\
        \sum\nolimits_j\tau_{ij}=h_i
	\end{eqnarray}
\end{subequations}

As a result, \eqref{case_1_2} defines the robust FR (RFR) that is robust to uncertain $E$, and maximising $r(p_1)$ over $(p_1,q,\tau_{ij})$ in this RFR will report the desired RDOEs. The final optimisation problem with the objective maximising the total DOE can be formulated as $\max_{p_1,q}\{r(p_1)|s.t.~\eqref{case_1_2}\}$.

\subsubsection{With uncertainty only in $p_2$}\label{sec_uset_b}
With $vec(E)$ fixed at $\bar e$ and denoting $g_i=A_2^TH_i^T\bar e$, for the $i^{th}$ constraint in \eqref{fr_new}, reformulating \eqref{fr_new} while considering uncertainty in $p_2$ leads to
\begin{subequations}\normalsize
	\begin{eqnarray}
		\max_{p_2\in\mathcal{P}}{g_i^Tp_2}=\delta^*(g_i|\mathcal{P})\nonumber\\
		=\min_{\phi_{i1},\phi_{i2}}\{\delta^*(\phi_{i1}|\mathcal{P}_1)+\delta^*(\phi_{i2}|\mathcal{P}_2))\big|\phi_{i1}+\phi_{i2}=g_i\}\nonumber\\
		\label{rfr_b_delta}
		=\min_{\phi_{i1},\phi_{i2}}\{\sum\nolimits_j(u_j^T\phi_{ij}+\delta^*(U_j^T\phi_{ij}|\mathcal{Y}_j))\big|\sum\nolimits_j\phi_{ij}=g_i\}\\
		\label{rfr_b}
		=\min_{\phi_{i1},\phi_{i2}}\{\sum\nolimits_j u_j^T\phi_{ij}+\rho_1||U_1^T\phi_{i1}||_1+\rho_2||U_2^T\phi_{i2}||_*)\nonumber\\\big|\sum\nolimits_j\phi_{ij}=g_i\}
		\le t_i-\bar e^TH_i(A_1p_1+Bq-b)
	\end{eqnarray}
\end{subequations}
where $\mathcal{Y}_1=\{y|\norm{y}_\infty\le\rho_1\}$ and $\mathcal{Y}_2=\{y|\norm{y}\le\rho_2\}$.

Similar to the derivation of the RFR with uncertain $E$, removing the \emph{min} operator in \eqref{rfr_b} also leads to an RFR that is robust against uncertain $p_2$. The final optimisation problem to maximise the total DOE can thus be formulated as
\begin{subequations}\label{max_oe_p2}
	\begin{eqnarray}
		\max_{p_1,q}r(p_1)\\        u_1^T\phi_{i1}+u_2^T\phi_{i2}+\rho_1||U_1^T\phi_{i1}||_1+\rho_2||U_2^T\phi_{i2}||_*\nonumber\\
			\le t_i-\bar e^TH_i(A_1p_1+Bq-b)~\forall i\\
		\phi_{i1}+\phi_{i2}=g_i~\forall i
	\end{eqnarray}
\end{subequations}

\color{black}
\subsubsection{With uncertainties in both E and $p_2$}\label{sec_uset_Eb}
In this case, bilinear uncertainty exists in \eqref{fr_new}, making the RC reformulation generally intractable. However, as discussed in \cite{Bertsimas2022}, a tractable reformulation is achievable when the uncertainty set follows specific types. One case is when $\mathcal{P}$ is formulated as 
\begin{eqnarray}\label{bi_ucty}
	\mathcal{P}=\mathcal{P}_1\cap\mathcal{P}_2
	=\{p_2|u+Uy,y\in\mathcal{Y}=\mathcal{Y}_1\cap\mathcal{Y}_2\}
\end{eqnarray}
where $u=u_1=u_2$, $U=U_1=U_2$, $\mathcal{Y}_1=\{y|\norm{y}_\infty\le\rho_1\}$ and $\mathcal{Y}_2=\{y|\norm{y}_1\le n_t\rho_1\}$ with $n_t\le n$, and $n$ is the cardinality of $y$.

Obviously, there is a total number of $2^{n_t}\tbinom{n}{n_t}$ extreme points in $\mathcal{Y}$. As a special case, when $n_t=1$, the $2n$ extreme points in $\mathcal{Y}$ can be expressed as $\{\pm y_1, \cdots,\pm y_k, \cdots, \pm y_n\}$, where $y_k\in\mathbb{R}^{n\times 1}$ is a vector with the $k^{th}$ element being $\rho_1$ and all the other elements being 0. Based on \eqref{rfr_E}, reformulating the $i^{th}$ constraint in \eqref{fr_new} while considering uncertainty in both $E$ and $p_2$ leads to
\begin{subequations}
	\begin{eqnarray}\label{rfr_Eb}
		\min_{\tau_{i1,k},\tau_{i2,k}}\{\sum\nolimits_j e_j^T\tau_{ij,k}+\gamma_1||J_1^T\tau_{i1,k}||_1+\gamma_2||J_2^T\tau_{i2,k}||_*\nonumber\\\big|\sum\nolimits_j\tau_{ij,k}=h_i(y_k)\}
		\le t_i~\forall k~~\\
		\min_{\lambda_{i1,k},\lambda_{i2,k}}\{\sum\nolimits_j e_j^T\lambda_{ij,k}+\gamma_1||J_1^T\lambda_{i1,k}||_1+\gamma_2||J_2^T\lambda_{i2,k}||_*\nonumber\\\big|\sum\nolimits_j\lambda_{ij,k}=h_i(-y_k)\}
		\le t_i~\forall k~~
	\end{eqnarray}
\end{subequations}
where $h_i(\pm y_k)=H_i(A_1p_1+A_2u \pm A_2Uy_k+Bq-b)$.

Similarly, the final equivalent deterministic formulation to maximise the total DOE can be formulated as
\begin{subequations}\label{max_oe_Ep2}
	\begin{eqnarray}
		\max_{p_1,q}r(p_1)\\
		\sum\nolimits_j e_j^T\tau_{ij,k}+\gamma_1||J_1^T\tau_{i1,k}||_1+\gamma_2||J_2^T\tau_{i2,k}||_*\nonumber\\\le t_i
				~\forall i,\forall k \\
		\sum\nolimits_j e_j^T\lambda_{ij,k}+\gamma_1||J_1^T\lambda_{i1,k}||_1+\gamma_2||J_2^T\lambda_{i2,k}||_*\nonumber\\\le t_i
				~\forall i,\forall k \\
		\sum\nolimits_j\tau_{ij,k}=h_i(y_k)~\forall i,\forall k\in\{1,2,\cdots\}\\
		\sum\nolimits_j\lambda_{ij,k}=h_i(-y_k)~\forall i,\forall k\in\{1,2,\cdots\}
	\end{eqnarray}
\end{subequations} 

Several remarks on calculating RDOEs are given below.
\begin{enumerate}
	\item In this letter, a \emph{strictly equal} allocation strategy, i.e. DOEs of all active customers being equal to each other, will be used, leading to a linear formulation of the objective function: $r(p_1)=1^Tp_1$ subject to $p_{1,i}=p_{1,j}~(\forall i\neq j)$. However, other objective functions can also be applied.
	\item Noting that \eqref{fr_new} is linear in $E$, in $p_2$ and in $q_2$, a tractable RC for this constraint can always be derived with single uncertainty, i.e. when uncertainty appears only in $E$ or only in $p_2$, and if the uncertainty set is convex. When bilinear uncertainty exists, a tractable RC formulation is achievable if there is a finite number of extreme points for at least one uncertainty set, as shown in \eqref{rfr_Eb}. However, enumerating all extreme points itself can be difficult.
	\item The problem \eqref{max_oe} with single uncertainty becomes a linear programming (LP) problem when 1-norm or $\infty$-norm is used in \eqref{uset}, and becomes second-order cone programming (SOCP) problem when 2-norm is used. However, the RC of \eqref{max_oe} with single uncertainty is always a convex programming problem if the uncertainty set is convex.
	\item Compared with \eqref{fr_new}, a \emph{buffer} term is added to the left side of \eqref{fr_new} in its RC, leading to enhanced \emph{robustness} of the solution. When $\gamma_i$, $\rho_i$, or $\theta_i$ equals 0, the RC deteriorates to the deterministic formulation.
	\item Due to the unbalances and mutual couplings of all phases in a distribution network and the fact that the active power of a VPP customer may vary between 0 kW and its allocated DOE, there is another type of uncertainty related to the difference between the optimal solution of $p_2$ and its realised value $\hat p_2$. Although this is beyond the scope of this letter, such uncertainty can be addressed by: 1) taking the approach proposed in \cite{liu2022robust} on top of an RFR based on the approach proposed in this letter; 2) employing the strictly equal allocation strategy, which is already taken in this letter, noting that such an allocation strategy, even with a deterministic DOE calculation approach, can provide satisfactory robustness, as demonstrated in \cite{liu2022robust}.
	\item When extra constraints on $p_1$ and $q_1$ exist, an additional constraint $L_1p_1+L_2p_2\le r$ can be added to \eqref{fr} and \eqref{fr_new}.
\end{enumerate}

\section{Case Study}\label{case_study}
\subsection{Case setup}
Two distribution networks, one of which is a 2-bus illustrative network (TwbNetwork) and the other one is a real Australian network (AusNetwork), will be studied.
For the illustrative network, where its topology is presented in Fig.~\ref{fig_2_bus_topology}, an ideal balanced voltage source with the voltage magnitude being $1.0~p.u.$ is connected to bus 1. A three-phase line connects bus 1 and bus 2, and its impedance matrix can be found in \cite{liu2022robust}.
\begin{figure}[htpb!]
	\centering\includegraphics[scale=0.1]{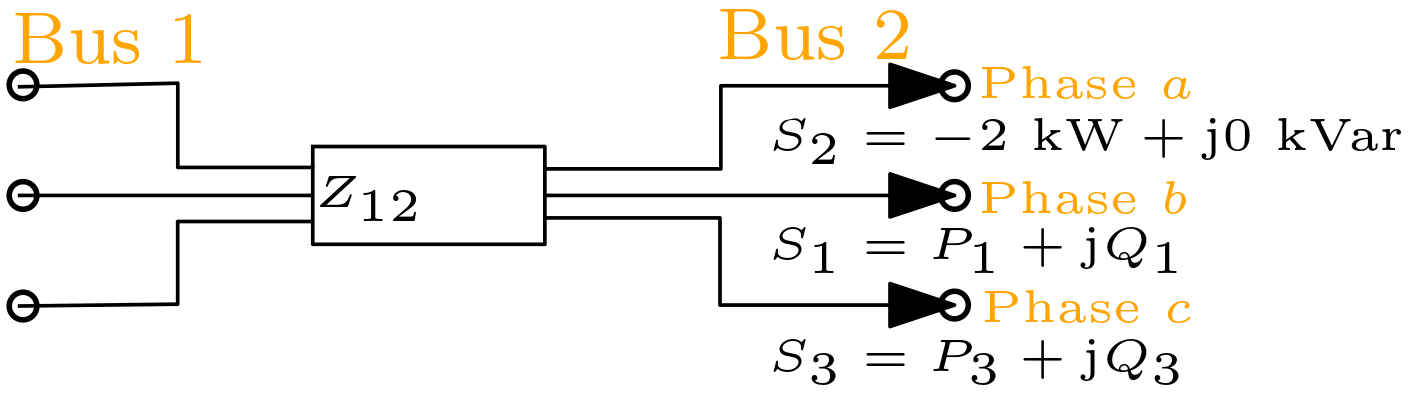}
	\caption{\small Network topology of the 2-bus illustrative example.}
	\label{fig_2_bus_topology}
\end{figure}

Of the three customers, $S_2=P_2+\text{j}Q_2$ is fixed while $P_1, P_3, Q_1$ and $Q_3$ are to be optimised with $r(P)=-P_1-P_3$ with $P_1=P_3$, aiming at maximising the total exports from customers 1 and 3. Moreover, the default export/import limits for both customers are set as 7 kW, and controllable reactive powers are assumed to be within [-1 kVar, 1 kVar]. Lower and upper voltage magnitude limits are set as $0.95~p.u.$ and $1.05~p.u.$, respectively.
The Australian network has 33 buses and 87 customers, of which 30 are VPP participants whose DOEs are to be calculated. For the remaining 57 customers, their reactive powers are fixed, while the active powers are treated as uncertain parameters. The default limits on active and reactive powers are the same as those in the illustrative network, and other data can be found in \cite{liu2022robust}.

For network impedances, $x$ in \eqref{uset_E} refers to the mutual impedances of line $12$ for the illustrative network and refers to the positive, negative and zero-sequence impedances of all line codes of lines ``46-47", ``69-67", ``49-50", ``40-41", ``54-59", ``45-50", ``67-68", ``44-45", ``61-62" and ``52-54" for the Australian network.

\color{black}
\subsection{Errors from the linearised model}
This section will investigate the accuracy of the employed linearised model based on the AusNetwork, where the given voltage for phase $a,b$ and $c$ are respectively set as $1.0\angle 0^\circ$,  $1.0\angle -120^\circ$ and  $1.0\angle 120^\circ$ for all buses. 

The average and maximum voltage magnitude (VM) errors when the demand for each of the active customers is at 1 kW (low customer load) and 3 kW (high customer load)\footnote{The value 3 kW is taken noting that RDOE calculated for each of the customers is around 3 kW.}, under both exporting and importing statuses, are presented in Table \ref{fig_linear_accuracy}. When customers' demands are at a high level, the average and maximum errors are respectively around 0.23\% and 0.59\% when they are exporting powers to the grid, and are respectively at around 0.83\% and 1.78\% when they are importing powers, demonstrating that the linearisation approach can achieve acceptable accuracy for RDOE calculation. However, it is expected that errors will become more significant when true nodal voltages deviate from the given voltage points, which can occur when customers are exporting or importing powers at very high levels.  
\begin{table}[htbp]\footnotesize\renewcommand\arraystretch{0.98}
\centering
\setlength{\tabcolsep}{8 pt}
  \caption{\small The errors in voltage magnitude (VM) taking the linear UTPF model (customer high/low load: 3 kW/1 kW).}
    \begin{tabular}{c|c|c|c}
    \hline\hline
    \multirow{2}{*}{Customer Status} & \multirow{2}{*}{Customer Load} & \multicolumn{2}{c}{VM Error (p.u.)} \\
\cline{3-4}          &      & Avg.  & Max. \\
    \hline
    \multirow{2}{*}{Export} & High  & 0.002336     & 0.005877 \\
\cline{2-4}          & Low   &  0.000125     & 0.000298 \\
    \hline
    \multirow{2}{*}{Import} & High  & 0.008268     & 0.017820 \\
\cline{2-4}          & Low   &   0.001776     & 0.003675 \\
    \hline\hline
    \end{tabular}%
  \label{fig_linear_accuracy}%
\end{table}%

The nodal voltages for all three phases throughout the whole network are also presented in Fig.\ref{fig_lin_errors_exp} when all active customers are exporting at 3 kW, which clearly shows that the voltages calculated by the linearised unbalanced three-phase power flow (LIN-UTPF) and by the non-convex UTPF (NCVX-UTPF) are very close to each other under this scenario. 
\begin{figure}[htpb!]
	\centering\includegraphics[scale=0.32]{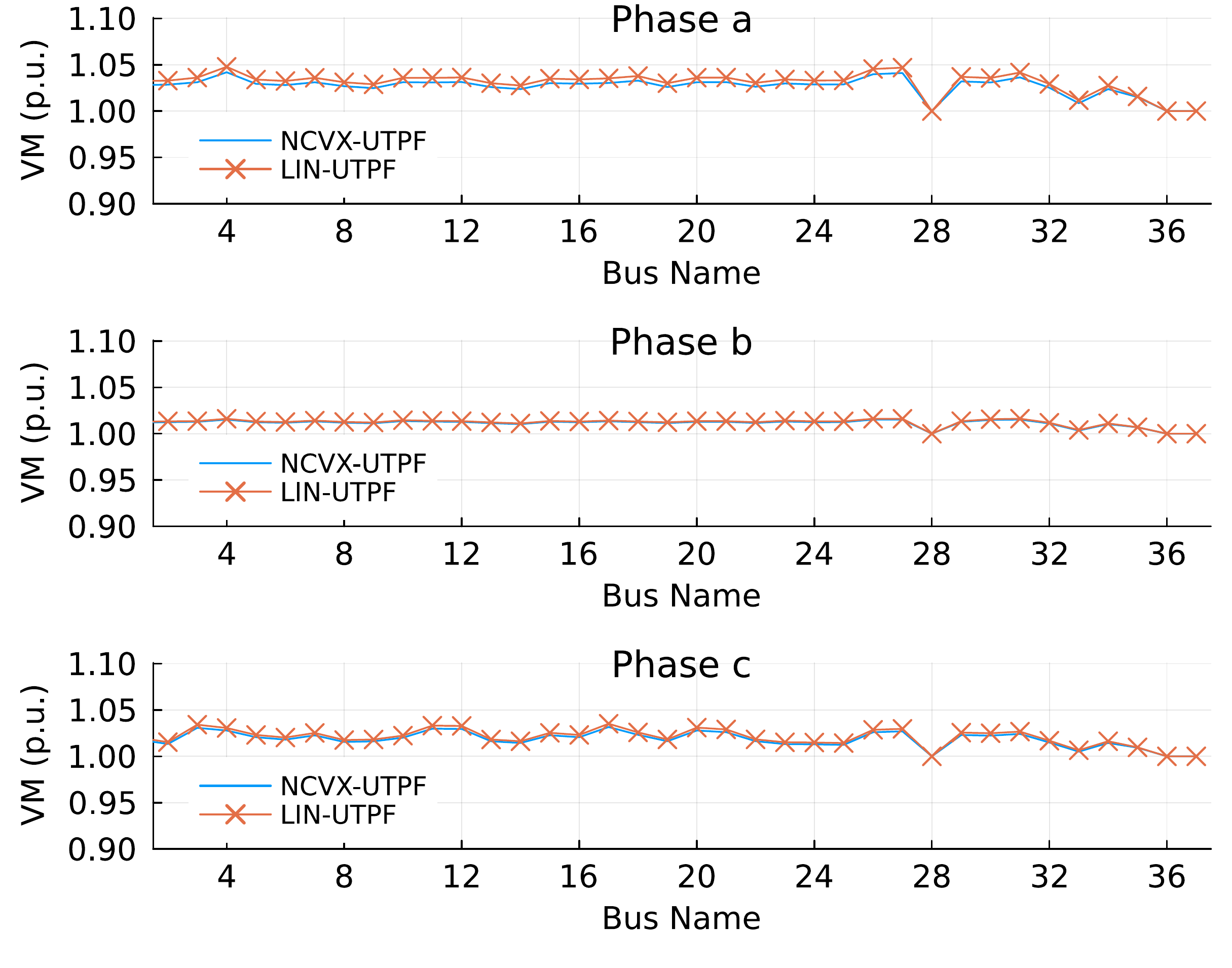}
	\caption{\small Nodal voltage magnitude under the exact non-convex UTPF (NCVX-UTPF) and linearised UTPF (LIN-UTPF) for the AusNetwork, where all active customers are exporting at 3 kW.}
	\label{fig_lin_errors_exp}
\end{figure}

However, we admit that the errors brought by the linearisation approach are inevitable and, in some cases, may be high. Thus, more efforts are needed in this area. One of the approaches to improving the accuracy is by iteratively updating the \emph{given} voltage points used to linearise the model. Specifically, after solving the optimisation model with the optimal solutions of $p$ and $q$ as $p^*$ and $q^*$, the optimal solution for $v$ after this iteration can be expressed as follows based on \eqref{oe_utopf_com-01}-\eqref{oe_utopf_com-02}. 
    \begin{eqnarray}\label{v_sol}
        v^*=D^{-1}EC^{-1}(Ap^*+Bq^*-b)+D^{-1}d
    \end{eqnarray} 

Then, matrix $C$, which depends on the given voltage points, can be updated further, followed by the re-calculation of the RDOE. The effectiveness of the iteration-based approach has been demonstrated in \cite{BL_cloud} and is omitted here for simplicity. 

\color{black}
\subsection{The calculated RDOEs}
Simulation results are presented in Table \ref{tab_case_study} and Fig.\ref{fig_2_bus_simu} with all optimisation problems solved by \texttt{Mosek} (version 10.0) \cite{mosek} on a laptop with Intel(R) Core(TM) i7-8550U CPU and 16 GB RAM, where the DOEs calculated by deterministic approach, denoted as DDOEs, are also presented for comparison purposes. \changed{Moreover, the compact formulation of the problems in this paper is realised with the assistance of \texttt{Julia} packages \texttt{MathOptInterface.jl}, \texttt{JuMP.jl} and \texttt{PowerModelsDistribution.jl} \cite{pmd_ref}.}
\begin{table}[htbp!]\footnotesize\renewcommand\arraystretch{0.98}
	\centering
	\setlength{\tabcolsep}{0.05 pt}
	\caption{\small DDOEs and RDOEs (negative values as \emph{export limit}, and values on the \emph{left/right} side in the table means: with $q_1$ being 0 kvar/with optimised $q_1$) and their computational times \changed{(including time for setting up the optimisation model)} under various uncertainties: A) $q_1=0$, $\mathcal{E}=\mathcal{E}_1$ ($\mathcal{E}_2$ not considered) with $e_1=vec(\bar E)$ and $J_1=diag(e_1)$, and B) $\mathcal{P}=\mathcal{P}_2$ ($\mathcal{P}_1$ not considered) with $u_2=\bar p_2$ and $U_2=diag(u_2)$.}
	\begin{tabular}{c|c|c|c|c|c}
		\hline\hline
		\multirow{2}{*}{Uncer.}   & \multirow{2}{*}{Norm}                                                                                                                                           & \multicolumn{2}{c|}{Optimal Obj. (kW)}                                                                                & \multicolumn{2}{c}{Time ($\times 10^{-2}$seconds)}                           \\	\cline{3-6}
		                            &                                                                                                                                                                  & TwbNetwork                                                  & AusNetwork                                               & TwbNetwork                                                & AusNetwork \\	\hline
		Deter.                      & N/A                                                                                                                                                              & -9.9/-13.7                    & -78.7/-112.3                          & 0.7/0.8                    & 3.7/4.9       \\	\hline
		\multirow{3}{*}{$E$}     & $\infty(\gamma_1=5.0\%$)                                                                                           & -9.2/-12.8        & -49.2/-58.3              & 4.3/5.4        & 36.0/57.8       \\
		                            & $\infty(\gamma_1=10.0\%$)                                                                                           & -8.6/-11.9        & -36.3/-43.1              & 3.9/5.3        & 36.2/61.6       \\	\hline
		\multirow{3}{*}{$p_2$}   & 1($\rho_2=20.0\%$)                                                                                                          & -9.4/-13.2                      & -78.1/-111.7                            & 1.0/1.1                      & 44.2/58.3      \\
		                            & 2($\rho_2=20.0\%$)                                                                                                          & -9.4/-13.2                      & -77.8/-111.4                            & 1.1/0.8                      & 20.2/36.1      \\
		                            & $\infty(\rho_2=20.0\%$)                                                                                                  & -9.4/-13.2                  & -75.4/-109.0                        & 0.7/1.0                  & 68.9/162.6      \\	\hline
		$(E,p_2)$                 & \makecell{$(\infty,1)$ \\($\gamma_1=5.0\%$,\\ $\rho_2=20.0\%$)}     & -8.7/-12.3    & -48.7/-56.1  & 30.0/31.1    & 2856.0/5839.0      \\
		\hline\hline
	\end{tabular}
	\label{tab_case_study}
\end{table}

\begin{figure}[htpb!]
	\centering
	\begin{subfigure}[b]{0.24\textwidth}
		\centering
		\includegraphics[width=\textwidth]{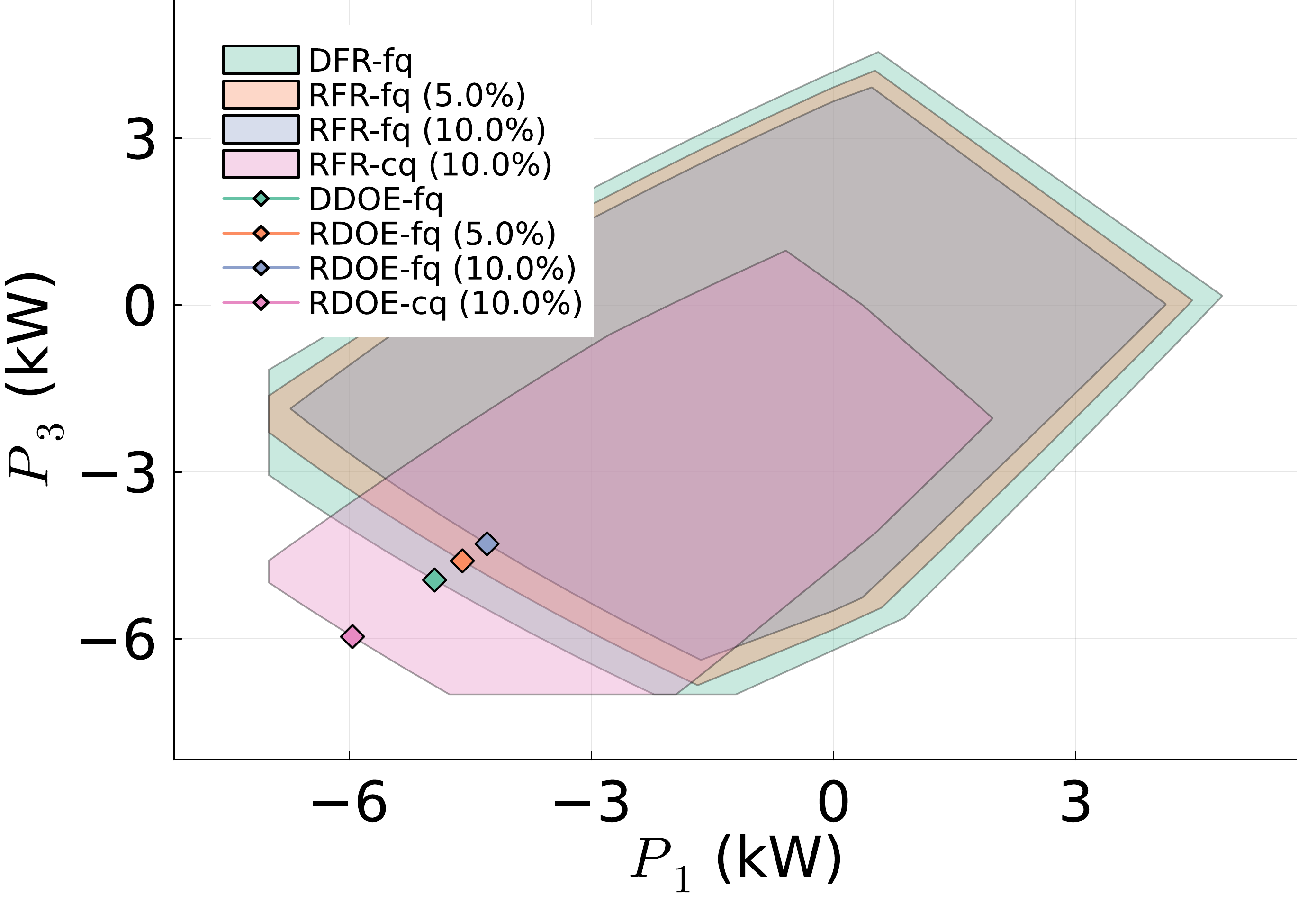}
		\caption{\small Uncertainty only in $E$.}
		\label{fig_FRs_E_2_bus}
	\end{subfigure}
	\hfill
	\begin{subfigure}[b]{0.24\textwidth}
		\centering
		\includegraphics[width=\textwidth]{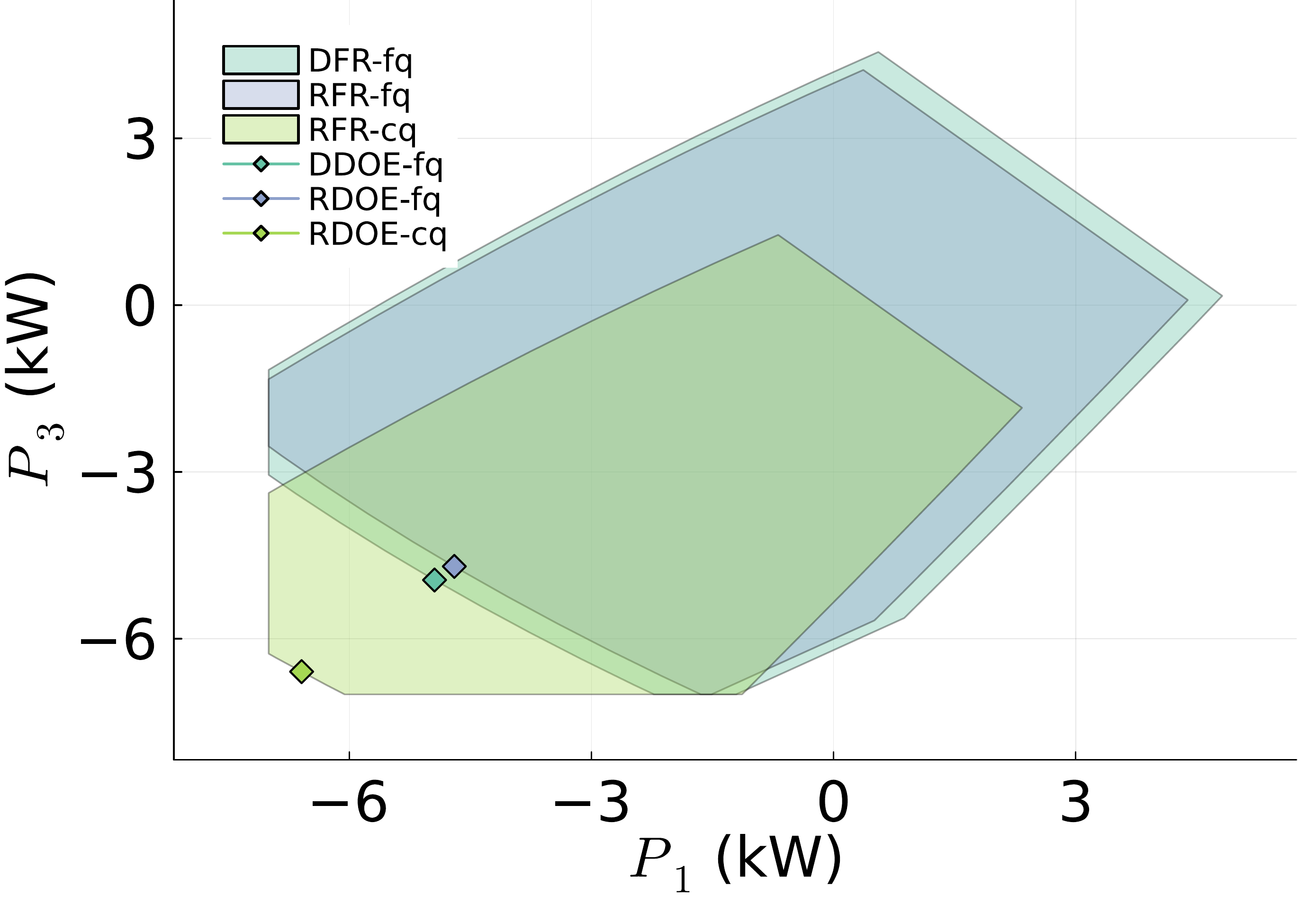}
		\caption{\small Uncertainty only in $p_2$.}
		\label{fig_FRs_p2_2_bus}
	\end{subfigure}
	\hfill
	\begin{subfigure}[b]{0.24\textwidth}
		\centering
		\includegraphics[width=\textwidth]{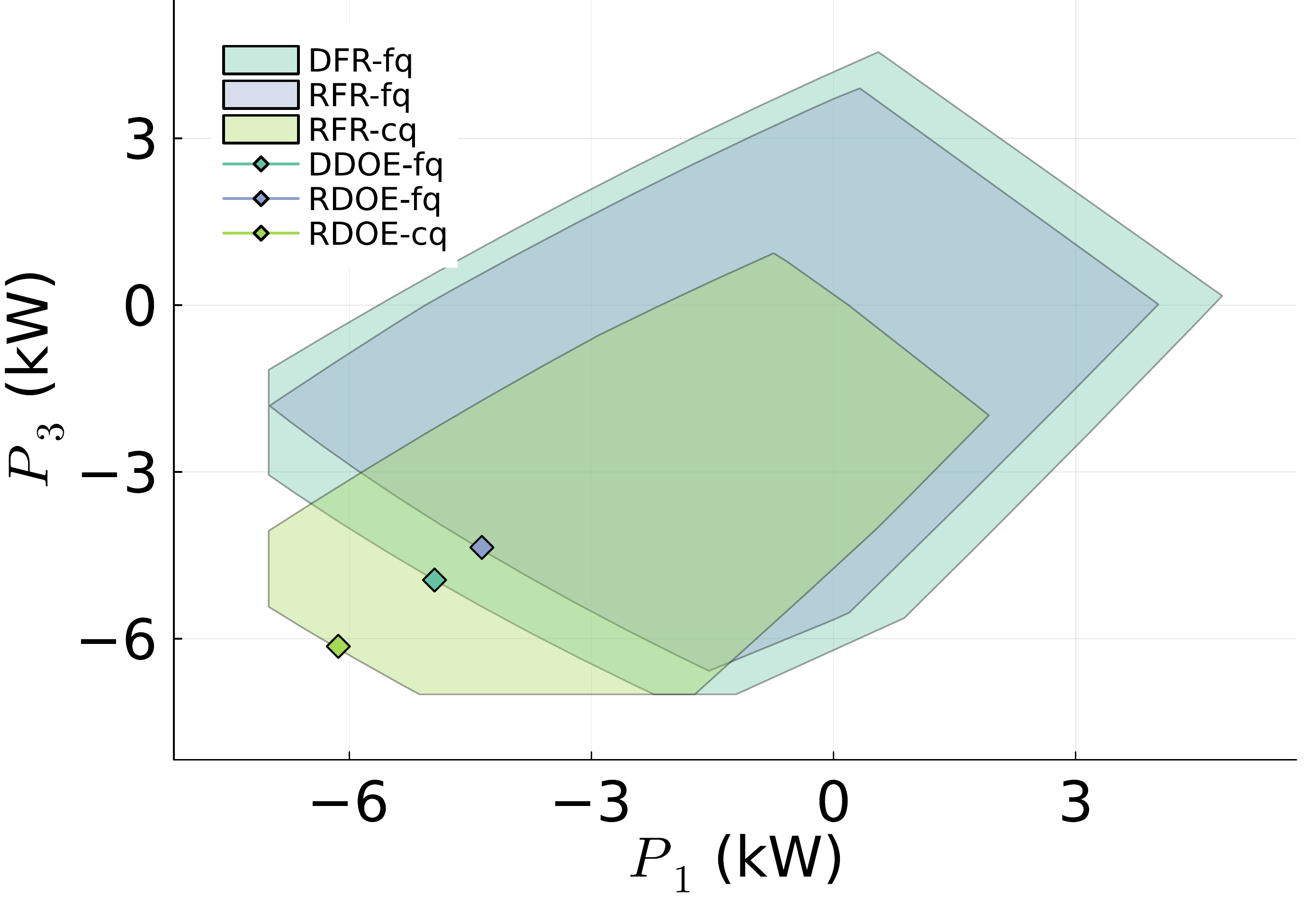}
		\caption{\small Uncertainty in both $E$ and $p_2$.}
		\label{fig_FRs_Ep2_2_bus}
	\end{subfigure}
	\hfill
	\begin{subfigure}[b]{0.24\textwidth}
		\centering
		\includegraphics[width=\textwidth]{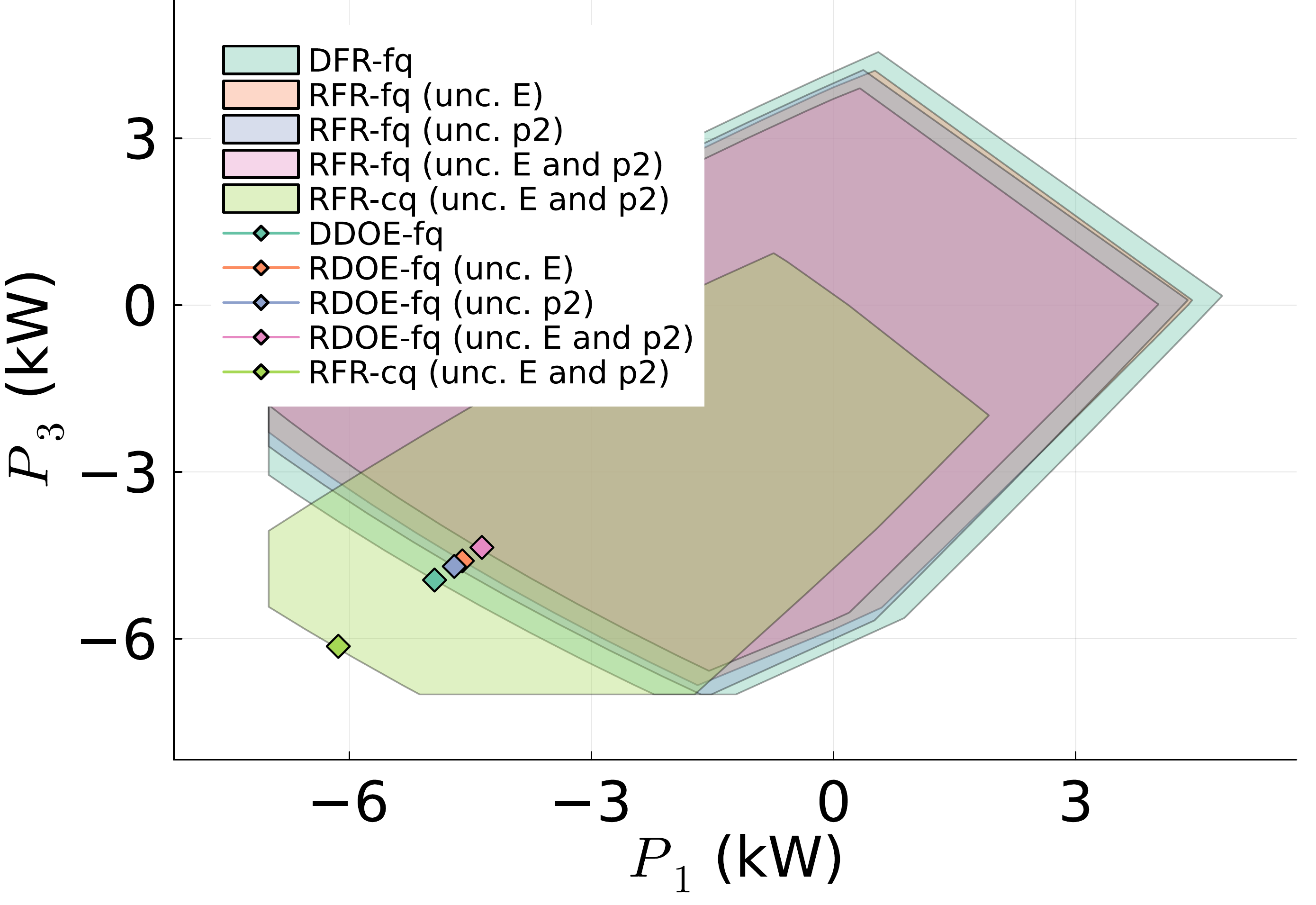}
		\caption{\small Under various uncertainties.}
		\label{fig_FRs_various_2_bus}
	\end{subfigure}
	\hfill
	\caption{\small FRs and DOEs for the illustrative network under various uncertainties (DFR/RFR: FR via deterministic/robust approach; DDOE/RDOE: deterministic/Robust DOE; \emph{fq}: with fixed $q_1$ (values being 0 kvar), \emph{cq}: with optimised $q_1$).}
	\label{fig_2_bus_simu}
\end{figure}

Simulation results clearly show that RDOEs are more conservative than DDOEs, and a higher level of uncertainty leads to a more conservative allocation strategy, as demonstrated in Fig.\ref{fig_FRs_E_2_bus}. Moreover, as shown in Table \ref{tab_case_study} and Fig.\ref{fig_2_bus_simu}, RFRs and allocated DOEs with optimised controllable reactive powers can effectively report ameliorated DOEs. On computational time, RDOEs can be calculated efficiently for both networks except when bilinear uncertainty exists in the Australian network. Moreover, setting up the optimisation model for this case can also be computationally demanding. In fact, for the Australian network, it takes at most 9 seconds to set up the optimisation model when there is a single uncertainty. In comparison, the setup time is as high as 2306.05 seconds when bilinear uncertainty exists due to a large number of constraints in \eqref{rfr_Eb}, implying computational efficiency can be potentially improved by investigating efficient programming techniques.

\section{Conclusions}
Uncertainties in demand forecasting and impedance modelling in distribution networks are inevitable and could potentially undermine the reliability of calculated DOEs for DER integration. This letter studies the calculation of DOEs when single or bilinear uncertainty exists in demands and network impedances, leading to various tractable formulations. Moreover, uncertainty sets are formulated as generalised norm constraints and could cover the most commonly used measures in quantifying uncertainties. 
Simulation results show the differences in DOE allocation strategies geometrically when with and without considering uncertainties, and demonstrate the efficiency of the proposed approach. Noting that the proposed approach is built on a linear UTOPF model, further improving accuracy in linearising UTOPF and investigating robust formulations under other types of uncertainty sets are potential research directions.

\color{black}
\section{Appendix}
\subsection{Discussions on whether matrices $C$ and $D$ in \eqref{oe_utopf_com} are invertible}
We here take a three-bus illustrative distribution network\footnote{The network in Fig.\ref{fig_3_bus_system} is three-phase balanced. However, the conclusion also applies in an unbalanced distribution network, as we demonstrate later.} to demonstrate if and why the matrices $C$ and $D$ are generally invertible.
\begin{figure}[htpb!]
    \centering\includegraphics[scale=1.05]{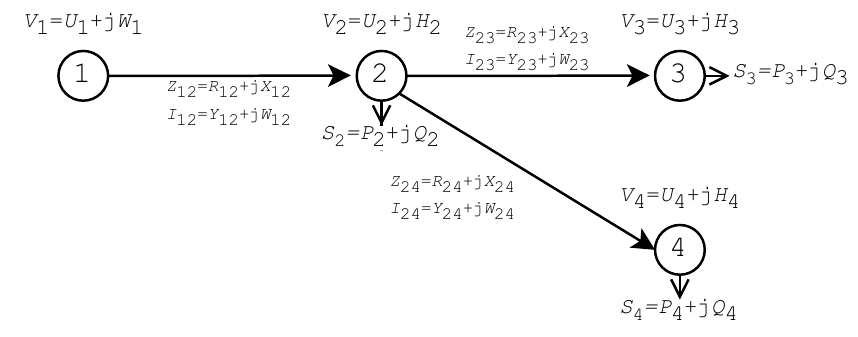}
    \caption{\normalsize Topology of the 3-bus illustrative network.}
    \label{fig_3_bus_system}
\end{figure}

\textbf{To show $C$ is invertible}

In the 3-bus illustrative network, by fixing the voltage $V_i^\phi$ (for a balanced distribution network this is $V_i$) in the denominator of the term on the right-hand side of (1d) in the manuscript, we have
    \begin{eqnarray}
        I_{12}-I_{23}-I_{24}=\frac{S_2^*}{\bar V_2^*},~
        I_{23}=\frac{S_3^*}{\bar V_3^*},~
        I_{24}=\frac{S_4^*}{\bar V_4^*}            
    \end{eqnarray}
where $\bar V$ represents the given value of $V$ for linearisation purposes, either from the estimation or network measurements, as we discussed in the manuscript. 

Then, we have 
\begin{eqnarray}\small\label{mtx_01}
    \begin{bmatrix}
        -1/\bar V_2^* & 0 & 0 \\0 & -1/\bar V_3^* & 0 \\0 & 0 & -1/\bar V_4^*            
    \end{bmatrix}
    \begin{bmatrix}
        S_2^*\\S_3^*\\S_4^*            
    \end{bmatrix}
    +\begin{bmatrix}
        1 & -1 & -1\\0 & 1 & 0 \\ 0 & 0 & 1            
    \end{bmatrix}
    \begin{bmatrix}
        I_{12}\\I_{23}\\I_{24}            
    \end{bmatrix}=0
\end{eqnarray}
which, expressed by real numbers, is equivalent to
{\small\begin{eqnarray}\label{mtx_02}
    \begin{bmatrix}
        \frac{1}{\text{Re}(\bar V_2^*)} & 0 & 0 & \frac{1}{\text{Im}(\bar V_2^*)} & 0 & 0 \\
        0 & \frac{1}{\text{Re}(\bar V_3^*)} & 0 & 0 & \frac{1}{\text{Im}(\bar V_3^*)} & 0  \\
        0 & 0 & \frac{1}{\text{Re}(\bar V_4^*)}  & 0 & 0 & \frac{1}{\text{Im}(\bar V_4^*)} \\ 
        \frac{1}{\text{Im}(\bar V_2^*)} & 0 & 0 & \frac{-1}{\text{Re}(\bar V_2^*)} & 0 & 0 \\
        0 & \frac{1}{\text{Im}(\bar V_3^*)} & 0 & 0 & \frac{-1}{\text{Re}(\bar V_3^*)} & 0  \\
        0 & 0 & \frac{1}{\text{Im}(\bar V_4^*)}  & 0 & 0 & \frac{-1}{\text{Re}(\bar V_4^*)} \\       
    \end{bmatrix}
    \begin{bmatrix}
        P_2\\P_3\\P_4\\Q_2\\Q_3\\Q_4            
    \end{bmatrix}\nonumber\\
    -\begin{bmatrix}
        1 & -1 & -1 & 0 & 0 & 0\\
        0 & 1 & 0 & 0 & 0 \\ 
        0 & 0 & 1 & 0 & 0 \\  
        0 & 0 & 0 & 1 & -1 & -1\\
        0 & 0 & 0 & 0 & 1 & 0 \\ 
        0 & 0 & 0 & 0 & 0 & 1      
    \end{bmatrix}
    \begin{bmatrix}
        Y_{12}\\Y_{23}\\Y_{24}\\W_{12}\\W_{23}\\W_{24}            
    \end{bmatrix}=0
\end{eqnarray}}

Denoting the second matrix in \eqref{mtx_01} as $\bar C$, then the matrix $C_{bal}$ for the balanced distribution network can be constructed as $C_{bal}=\begin{bmatrix}\bar C & 0\\ 0 & \bar C\end{bmatrix}$, as shown by the second matrix in \eqref{mtx_02}, and $C_{bal}$ is invertible if $\bar C$ is invertible. Since $\bar C$ is the connectivity matrix between all nodes (excluding the reference bus) and all lines in the radial distribution network, it is always invertible and, thus $C_{bal}$ is also invertible. Similarly, for an unbalanced three-phase distribution network, $C$, the true matrix used in our formulation, can be constructed as $C=C_{bal}\otimes\begin{bmatrix} 1 & 0 & 0\\ 0 & 1 &0 \\0 & 0 & 1\end{bmatrix}$, where $\otimes$ is the Kronecker product operator. Thus, it is demonstrated that $C$ is also invertible.  

\textbf{To show $D$ is invertible}

For the 3-bus illustrative network, the expression (1c) can be expressed as 
{\small\begin{eqnarray}
    \bar V_1-V_2=Z_{12}I_{12},~
    V_2-V_3=Z_{23}I_{23},~
    V_2-V_4=Z_{24}I_{24}            
\end{eqnarray}}
where $\bar V_1=|\bar V_1|\angle 0^\circ$ is the fixed value for $V_1$ since it is the reference bus. 

Then, we have 
{\small\begin{eqnarray}\label{mtx_03}
    \begin{bmatrix}
        -1 & 0 & 0 \\1 & -1 & 0 \\1 & 0 & -1            
    \end{bmatrix}
    \begin{bmatrix}
        V_2\\V_3\\V_4            
    \end{bmatrix}
    +\begin{bmatrix}
        Z_{12} & 0 & 0 \\0 & Z_{23} & 0 \\ 0 & 0 & Z_{24}            
    \end{bmatrix}
    \begin{bmatrix}
        I_{12}\\I_{23}\\I_{24}            
    \end{bmatrix}
    =\begin{bmatrix}
        -\bar V_1\\0\\0            
    \end{bmatrix}
\end{eqnarray}}
which, expressed by real numbers, is equivalent to
{\small\begin{eqnarray}\label{mtx_04}
    \begin{bmatrix}
        -1 & 0 & 0 & 0 & 0 & 0\\
        1 & -1 & 0 & 0 & 0 & 0 \\
        1 & 0 & -1 & 0 & 0 & 0\\    
        0 & 0 & 0 & -1 & 0 & 0 \\
        0 & 0 & 0 & 1 & -1 & 0 \\
        0 & 0 & 0 & 1 & 0 & -1 \\        
    \end{bmatrix}
    \begin{bmatrix}
        U_2\\U_3\\U_4\\H_2\\H_3\\H_4            
    \end{bmatrix}+
    \nonumber\\\begin{bmatrix}
        R_{12} & 0 & 0 & -X_{12} & 0 & 0 \\
        0 & R_{23} & 0 & 0 & -X_{23} & 0\\ 
        0 & 0 & R_{24} & 0 & 0 & -X_{24}  \\       
        X_{12} & 0 & 0 & R_{12} & 0 & 0 \\
        0 & X_{23} & 0 & 0 & R_{23} & 0\\ 
        0 & 0 & X_{24} & 0 & 0 & R_{24}  \\ 
    \end{bmatrix}
    \begin{bmatrix}
        Y_{12}\\Y_{23}\\Y_{24}\\W_{12}\\W_{23}\\W_{24}            
    \end{bmatrix}
    =\begin{bmatrix}
        -|\bar V_1|\\0\\0\\0\\0\\0        
    \end{bmatrix}
\end{eqnarray}}

Denoting the first matrix in \eqref{mtx_03} as $\bar D$, which is invertible, it is obvious that $D_{bal}$, the first matrix in \eqref{mtx_04}, can be constructed as $D_{bal}=\begin{bmatrix}\bar D & 0 \\0 & \bar D\end{bmatrix}$, which is also invertible. Similarly for a unbalanced distribution network, $D$, the matrix used in our formulation, can be constructed as $D=D_{bal}\otimes\begin{bmatrix} 1 & 0 & 0\\ 0 & 1 &0 \\0 & 0 & 1\end{bmatrix}$, which is also invertible.

In general, both $C$ and $D$ in an unbalanced distribution network relate to the connectivity matrix for all buses (excluding the reference bus) and all lines in the radial network and can be guaranteed to be invertible.
    
\subsection{Discussions on using an exact AC power flow formulation to calculate RDOEs}
We here present and discuss the potential challenges when taking an exact AC power flow formulation to calculate RDOEs (taking the case when uncertainty only presents in $E$ as an example).
    
For this case, the compact form of the FR for calculating DOEs, instead of (3) in the manuscript, can be expressed as
\begin{equation}
      \mathcal{F}_N(q)
    \label{fr-03}
    =\left\{p\left\vert\begin{matrix}
        Ap+Bq+Ct(v,l)=b                \\
        Dv+El=d\\
        g(v)\le f                    \\
    \end{matrix}\right.\right\}
\end{equation}
where 
both $p=[p_1^T, p_2^T]^T$, $q=[q_1^T, q_2^T]^T$, and $t(v,l)$ and $g(v)$ are non-convex functions of $v$ when taking an exact power flow formulation that is non-convex.

As $\mathcal{F}_N$ cannot be expressed as a polyhedron, the robust formulation to seek DOEs that is immune to uncertain realisation of $E$ can be formulated as
\begin{subequations}\label{max_ellip_nvx}
    \begin{eqnarray}
        \label{max_ellip_obj_nvx}
        \max_{q,v,l}{r(p_1)}\\
        \label{max_ellip_cons_nvx}
        s.t.~~\exists (v,l) \Rightarrow \left\{\begin{matrix}
            Ap+Bq+Ct(v,l)=b                \\
            Dv+El=d\\
            g_i(v)\le f_i~~\forall i                    \\
        \end{matrix}\right.,~\forall E \in \mathcal{E}
    \end{eqnarray}
\end{subequations}

To solve the above problem, we need to reformulate the constraint into a deterministic form. Assuming $\mathcal{E}=\{E|vec(E)=e+Jx, ||x||\le\gamma\}$, for each $i$, constraint \eqref{max_ellip_cons_nvx} is equivalent to
\begin{subequations}\label{lag_01}
    \begin{eqnarray}
    \label{lag_01_obj}
        \max_{E,x}\min_{v,l}{g_i(v)}\le f_i\\
        \label{lag_01_cons_01}
        s.t.~~Ap+Bq+Ct(v,l)=b~~(\alpha_i)\\
        \label{lag_01_cons_02}
        Dv+El=d~~(\beta_i)\\
        vec(E)=e+Jx~~(\eta_i)\\
        ||x||\le \gamma~~(\delta_i)
    \end{eqnarray}
\end{subequations}

Nothing that the \emph{min} operator in \eqref{lag_01_obj} is based on the assumption that there might be multiple solutions for UTPF \cite{multi_solution}. To derive a deterministic formulation of \eqref{lag_01}, thus making the optimisation problem \eqref{max_ellip_obj_nvx} solvable, we need to remove both the \emph{max} and \emph{min} operators in \eqref{lag_01_obj}. Generally, removing the \emph{min} operator is challenging and, however, it can be naturally removed if we assume that $(v,l)$ are uniquely determined by \eqref{lag_01_cons_01}-\eqref{lag_01_cons_02}. Next, we will show how to remove the \emph{max} operator in \eqref{lag_01} under such an assumption based on duality theory in order to derive a deterministic formulation.

The Lagrangian function for \eqref{lag_01} (excluding the ``$\le f_i$"), when treading $p$ and $q$ as fixed variables and , is 
\begin{subequations}\small
    \begin{eqnarray}
    L(\alpha_i,\beta_i,\eta_i,\delta_i,E,x,v,l)\nonumber\\
        =g_i(v)+\alpha_i^T(Ap+Bq+Ct(v,l)-b)+\beta_i^T(Dv+El-d)\nonumber\\
                            +\eta_i^T(vec(E)-e-Jx)-\delta_i^T(||x||-\gamma)\\
    =g_i(v)+\alpha_i^T(Ap+Bq-b)+\alpha_i^TCt(v,l)+\beta_i^T(Dv-d)\nonumber\\+vec(E)^T(l\otimes \beta_i)+\eta_i^T(vec(E)-e-Jx)-\delta_i^T(||x||-\gamma)\\
    =g_i(v)+\alpha_i^T(Ap+Bq-b)+\alpha_i^TCt(v,l)+\beta_i^T(Dv-d)\nonumber\\+vec(E)^T(l\otimes \beta_i+\eta_i)-\eta_i^Te+\delta_i^T\gamma-\eta_i^TJx-\delta_i^T||x||
    \end{eqnarray}
\end{subequations}
and the optimisation problem \eqref{lag_01} is equivalent to
\begin{subequations}\label{lag_01_L}
    \begin{eqnarray}
    g_i(v)\le \min_{\alpha_i,\beta_i,\eta_i,\delta_i\ge 0}\max_{E,x}L(\alpha_i,\beta_i,\eta_i,\delta_i,E,x,v,l)\\
    =\left\{\begin{matrix}\min_{\alpha_i,\beta_i,\eta_i,\delta_i}g_i(v)+\alpha_i^T(Ap+Bq-b)+\alpha_i^TCt(v,l)\nonumber\\+\beta_i^T(Dv-d)-\eta_i^Te+\delta_i^T\gamma_i\\
    s.t.~~l\otimes \beta_i+\eta_i=0\\
    ||J^T\eta_i||_*\le\delta_i\end{matrix}\right.
\end{eqnarray}
\end{subequations}

By further removing the \emph{min} operator in \eqref{lag_01_L}, we have the following alternate formulation to solve \begin{subequations}\label{max_ellip_nvx_det}
    \begin{eqnarray}
        \label{max_ellip_obj_nvx_det}
        \max_{p,q,v,l,\alpha_i,\beta_i,\eta_i,\delta_i}{r(p_1)}\\
        \label{max_ellip_cons_nvx_det}
        s.t.~~g_i(v)+\alpha_i^T(Ap+Bq-b)+\alpha_i^TCt(v,l)\nonumber\\+\beta_i^T(Dv-d)-\eta_i^Te+\delta_i^T\gamma_i\le f_i~~\forall i\\
        l\otimes \beta_i+\eta_i=0~~\forall i\\
        ||J^T\eta_i||_*\le\delta_i~~\forall i
    \end{eqnarray}
\end{subequations}

Although the non-convex terms in \eqref{max_ellip_nvx_det} may be dealt with by \texttt{Ipopt} or other nonlinear solvers, the further introduced non-convexity due to $g_i(v), \alpha_i^T(Ap+Bq), \alpha_i^TCt(v,l)$ and $\beta_i^TDv$, and the assumptions/approximations made to derive \eqref{max_ellip_nvx_det} may substantially increase the computational complexity, or even lead to intractability, and also undermine the robustness in DOEs we want to achieve.

\subsection{Comparison of the proposed and the canonical three-stage robust optimisation (TSRO) approach}
We here provide more discussions on comparing the proposed approach and the canonical three-stage robust optimisation for making a robust decision against uncertainties. In general, there should be no difference between the solutions from the proposed approach and the TSRO. However, in terms of the specific formulations and the solution algorithms, the two approaches can be different, and the proposed approach can be potentially more efficient, as we explain next by taking the case where uncertainty only exists in $E$ as an example. 

Based on the formulation \eqref{oe_utopf_com} in the manuscript, the variables to be optimised in the first stage\footnote{For simplicity, we assume $p_1,p_2,q_1$ and $q_2$ are variables to be optimised.} are $p,q$, in the second stage is $E$, and in third stage are $v,l$. Then the TSRO can be formulated as
\begin{subequations}\small\label{oe_utopf_com_ro}
    \begin{eqnarray}
        \label{oe_utopf_com_ro-obj}
        \max_{p,q}{r(p)}\\
        \label{oe_utopf_com_ro-01}
        s.t.~~\forall E\in \mathcal{E}, \exists (v,l)\Rightarrow\left\{\begin{matrix}Ap+Bq+Cl=b\\Dv+El=d\\Fv\le f\end{matrix}\right.
    \end{eqnarray}
\end{subequations}

The above problem can be solved as follows based on the widely used robust optimisation approach (see examples when applied in power system operation in \cite{ro_uc_01,ro_uc_02}). 
\begin{enumerate}[A)]
    \item Initialise $\mathbb{E}=\{E_0\}$, where $E_0$ is the expected value of $E$ and $E_0\in\mathcal{E}$. 
    \item Solve the following optimisation problem to get the optimal DOEs, say $p^*$. 
    \begin{subequations}\small\label{oe_utopf_com_0}
        \begin{eqnarray}
            \label{oe_utopf_com-obj_0}
            \max_{p,q}{r(p)}\\
            \label{oe_utopf_com-01_0}
            s.t.~~Ap+Bq+Cl_i=b\\
            Dv_i+E_i l_i=d~\forall E_i\in\mathbb{E}\\
            Fv_i\le f
        \end{eqnarray}
    \end{subequations}

    \item Solve the following optimisation problem\footnote{Since the voltage at the reference bus is known, and $v$ and $l$ can be uniquely determined when $p,q$ and $E$ are fixed, the optimisation problem is a single \emph{max} problem instead of a \emph{max-min} problem that is reported in some other formulations, including, for example, the ones in \cite{ro_uc_01,ro_uc_02}.} and record the optimal solution of $E$ as $E^*$. 
    \begin{subequations}\small\label{oe_utopf_com_01}
        \begin{eqnarray}
            \label{oe_utopf_com-obj_01}
            \max_{E\in\mathcal{E}}{1^T(t^++t^-)+1^T(s^++s^-)+1^Tu}\\
            \label{oe_utopf_com-01_01}
            Ap+Bq+Cl+t^+-t^-=b\\
            Dv+El+s^+-s^-=d\\
            Fv-u\le f\\
            t^+\ge 0, t^-\ge 0, s^+\ge 0, s^-\ge 0, u\ge 0
        \end{eqnarray}
    \end{subequations}

    If the optimal objective value is positive, then update $\mathbb{E}=\mathbb{E}\cup\{E^*\}$ and go back to Step (B). Otherwise, $p^*$ calculated from Step (C) is the desired RDOEs.
\end{enumerate}

Obviously, the approach based on TSRO needs to alternately solve two linear programming (LP) problems until the optimal objective value of \eqref{oe_utopf_com_01} is 0. By contrast, the proposed approach in this paper first reformulates the constraints in \eqref{oe_utopf_com-01} to remove state variables $v$ and $l$, leading to $FD^{-1}EC^{-1}(Ap+Bq-b)\le f-FD^{-1}d$, where only decision variable $p$ and $q$ present. Then depending on the specific type of uncertainty to be investigated and its formulation, various equivalent reformulations can be derived to simplify the calculation process compared with the TSRO approach. 

\color{black}
\bibliography{REFs_Power_Grid}

\bigskip
\footnotesize{\noindent\textbf{Bin Liu} received his Bachelor, Master and PhD degrees, all in Electrical Engineering, from Wuhan University, Wuhan, China, China Electric Power Research Institute, Beijing, China, and Tsinghua University, Beijing, China, in 2009, 2012 and 2015, respectively. He is a Senior Power System Engineer in the Network Planning Division, Transgrid, Sydney, Australia. Before joining Transgrid, he had held research or engineering positions with the Energy Centre, Commonwealth Scientific and Industrial Research Organisation (CSIRO), Newcastle, Australia, The University of New South Wales, Sydney, Australia, the State Grid, Beijing, China, and The Hong Kong Polytechnic University, Hong Kong. His current research interests include power system modelling, analysis \& planning, optimisation theory applications in the power and energy sector, and the interaction of renewable energy, including distributed energy resources (DERs).}\\

\footnotesize{\noindent\textbf{Julio H. Braslavsky} received his PhD in Electrical Engineering from the University of Newcastle NSW, Australia in 1996, and his Electronics Engineer degree from the National University of Rosario, Argentina in 1989. He is a Principal Research Scientist with the Energy Systems Program of the Australian Commonwealth Scientific and Industrial Research Organisation (CSIRO) and an Adjunct Senior Lecturer with The University of Newcastle, NSW, Australia. He has held research appointments with the University of Newcastle, the Argentinian National Research Council (CONICET), the University of California at Santa Barbara, and the Catholic University of Louvain-la-Neuve in Belgium. His current research interests include modelling and control of flexible electric loads and integration of distributed power-electronics-based energy resources in power systems. He is Senior Editor for IEEE Transactions on Control Systems Technology.}\\

\footnotesize{\noindent\textbf{Nariman Mahdavi} received the Ph.D. (First Class) degree in electrical engineering from the Amirkabir University of Technology, Tehran, Iran, in 2011. He is a Senior Research Scientist with Power Systems and Controls, Energy Centre, CSIRO, Australia. Prior to joining CSIRO in 2017, he was a Postdoctoral Researcher with the University of Newcastle, Australia, in collaboration with CSIRO Energy Centre, and held research appointments with the Potsdam Institute for Climate Impact Research, Germany. His research focuses on mathematical modelling, analysis, estimation and control of dynamical systems to help navigate through the transformation across the energy sector.}

\end{document}